\documentclass[a4paper,12pt,reqno]{amsart}
\usepackage{amsmath,amsfonts,amssymb,amsthm,enumerate,multicol}%hyperref,
\usepackage{tikz,graphicx}
\usepackage{hyperref}
\usepackage{caption,enumitem,wasysym}
\usepackage{cleveref}
\usepackage{epstopdf}
\usepackage{float,wrapfig}
\usepackage[labelsep=space]{caption}
\usepackage[font=footnotesize]{caption}
\usepackage{xcolor}
\usepackage{graphicx}   
\newtheorem{theorem}{Theorem}[section]

\newtheorem{definition}[theorem]{Definition}
\newtheorem{corollary}[theorem]{Corollary}

\newtheorem{lemma}[theorem]{Lemma}
\newtheorem{example}[theorem]{Example}
\newtheorem{remark}[theorem]{Remark}

\allowdisplaybreaks
\theoremstyle{plain}
\newtheorem{thm}{Theorem}[section]

\newtheorem{prop}[thm]{Proposition}

\usepackage{mathtools}

\usepackage{amsmath,amsfonts,amssymb,amsthm,enumerate,multicol}%hyperref,
\usepackage{tikz}
\usepackage{float}
\allowdisplaybreaks

\numberwithin{equation}{section}

\begin{document}
\begin{center}
\Large{\textbf{The  discrete collision-induced breakage equation with mass transfer: well-posedness and stationary solutions }}
\end{center}

%\begin{abstract}

%\end{abstract}

%\vskip 11pt \hrule

%%%%%%%%%%%%%%%%%%%%%%%%%%%%%%%%%%%%%%%%%%%%%%%%%%%%%%%%%%%
%%%%%%%%%%%%%%%%%%%%%%%%%%%%%%%%%%%%%%%%%%%%%%%%%%%%%%%%%%%

\medskip
%\centerline{by}
\medskip
\centerline{${\text{Mashkoor~ Ali$^*$}}$, ${\text{Ankik ~ Kumar ~Giri$^{\dagger}$}}$ and ${\text{Philippe~ Lauren\c{c}ot$^{\ddagger}$}}$}\let\thefootnote\relax\footnotetext{$^{\dagger}$Corresponding author. Tel +91-1332-284818 (O);  Fax: +91-1332-273560  \newline{\it{${}$ \hspace{.3cm} Email address: }}ankik.giri@ma.iitr.ac.in}
\medskip
{\footnotesize
 %please put the address of the second  and third author

  \centerline{ ${}^{}$  $*$,$\dagger$ Department of Mathematics, Indian Institute of Technology Roorkee,}
   %\centerline{Other lines}
   \centerline{Roorkee-247667, Uttarakhand, India}

 \centerline{ ${}^{}$ $\ddagger$ Laboratoire de Math\'ematiques (LAMA) UMR 5127, Universit\'e Savoie Mont Blanc, CNRS,}
   %\centerline{Other lines}
   \centerline{ F-73000, Chamb\'ery, France}
}

\bigskip

\begin{quote}
{\small {\em \bf Abstract.} The discrete collisional breakage equation, which captures the dynamics of cluster growth when clusters undergo binary collisions with possible matter transfer, is discussed in this article. The existence of global mass-conserving solutions is investigated for the collision kernels  $a_{i,j}=A(i^{\alpha} j^{\beta} + i^{\beta}j^{\alpha})$, $i, j \ge 1$,  with $\alpha \in (-\infty,1)$, $\beta\in [\alpha,1]\cap (0,1]$, and $A>0$ and for a large class of possibly unbounded daughter distribution functions.  All  algebraic superlinear moments of these solutions are bounded on time intervals $[T,\infty)$ for any $T>0$. The uniqueness issue is further handled under additional restrictions on the initial data. Finally, non-trivial stationary solutions are constructed by a dynamical approach.}
\end{quote}

\vspace{.3cm}



\section{Introduction}
Coagulation and fragmentation processes are present in a variety of scientific and engineering fields, including chemistry (such as water vapor condensing into liquid droplets), astrophysics (planet formation), atmospheric science (raindrop breakup), and biology (red blood cell aggregation), see \cite{LG 1976, SV 1972, SRC 1978}. These processes involve the coalescence of particles to form larger ones and the breakage of particles into smaller pieces. Coagulation is always a nonlinear process, while breakage can be categorized into linear breakage, caused by external forces or spontaneous processes, and collisional or nonlinear breakage, resulting from particle collisions. Linear breakage only produces smaller particles, while collisional breakage can transfer mass between colliding particles, resulting in larger daughter particles. In this work, each particle's volume or size is represented by a natural number, i.e. the ratio of the mass of the basic building block (monomer) to the size of a typical cluster is positive, and so the size of a cluster is a finite multiple of the size of the monomer.

 More precisely, denoting by $w_i(t)$, $i \ge 1$, the number of clusters made of $i$ monomers ($i$-particles) per unit volume at time $t \geq 0$, the discrete collision-induced breakage equation reads
\begin{align}
\frac{dw_i}{dt}  =&\frac{1}{2} \sum_{j=i+1}^{\infty} \sum_{k=1}^{j-1} B_{j-k,k}^i a_{j-k,k} w_{j-k} w_k -\sum_{j=1}^{\infty} a_{i,j} w_i w_j,  \hspace{.5cm} i \in \mathbb{N}, \label{NLDCBE}\\
w_i(0) &= w_i^{\rm{in}}, \hspace{.5cm} i \in \mathbb{N},\label{NLDCBEIC}
\end{align}
where $\mathbb{N}$ stands for the set of positive integers. Here $a_{i,j}$ denotes the rate of collisions of $i$-clusters with $j$-clusters and satisfies
\begin{align}
a_{i,j}=a_{j,i} \geq 0,\qquad (i,j) \in \mathbb{N}^2, \label{ASYMM}
\end{align}
 while the daughter distribution function $\{B_{i,j}^s, 1 \le s\le i+j-1\}$  gives the information about the average number of particles with size $s$ generated during the breakage events that occur due to collisions between particles having sizes $i$ and $j$. The first term on the right hand side of equation \eqref{NLDCBE} represent the formation of particles with size  $i$ resulting from  the collisional breakage of a pair of particles with sizes $(j-k)$ and $k$, while the last term describes the decay of particles with size $i$ due to their collisions with particles with arbitrary sizes $j\ge 1$. We make the assumption that there is no loss of matter during the breakup process, and that no particle is created with a size greater or equal than the total size of the colliding particles. This implies that the daughter distribution function $B$ fulfills the following conditions: for $(i,j)\in\mathbb{N}^2$, 
 
 \begin{equation}
 	\begin{split}
		& B_{i,j}^s = B_{j,i}^s \geq 0 , \quad 1 \le s \le i+j-1, \hspace{.7cm} B_{i,j}^s = 0 \hspace{.7cm} \text{for} \hspace{.7cm} s \ge i+j, \\
		& \sum_{s=1}^{i+j-1} s B_{i,j}^s = i+j . \label{LMC}
	\end{split}
\end{equation}
Observe that \eqref{LMC} implies that
\begin{equation}
	B_{1,1}^1=2. \label{LMC11}
	\end{equation}

The discrete coagulation equation with collisional breakage has been analyzed in \cite{Laurencot 2001I}, where the existence of weak solutions that conserve mass is established by a compactness approach. Additionally, the authors have investigated the uniqueness of solutions, long-term behavior of the system, as well as the possibility of gelation transition under certain conditions on the collision kernel and daughter distribution function. In \cite{VVF 2006}, under the assumption that the collision kernel is constant, partial analytical solutions to \eqref{NLDCBE}--\eqref{NLDCBEIC} are obtained and compared against corresponding Monte-Carlo simulations. Furthermore, a comparison is made between the dynamics of the population balance equation involving coalescence and collision-induced breakage and the one coupling aggregation with  spontaneous (linear) fragmentation, highlighting their similarities and differences. Several mathematical existence results are also available for the continuous collision-induced fragmentation equation when it is combined with coagulation,  and the coagulation is generally assumed to be the dominating mechanism, see \cite{PKB 2020, PKB 2020I, AKG 2021}.

Even though no coagulation is accounted for in \eqref{NLDCBE}, the collision-induced breakup monitored by a daughter distribution function $B$ satisfying \eqref{LMC} may produce particles with a larger size than the incoming particles, as the formation of fragments with sizes ranging in $\big[ \max\{i,j\}+1, i+j-1\big]$ is not precluded by \eqref{LMC} (for instance, the collision of a particle with size $i$ with a particle with size $j$ may produce a monomer and a particle with size $i+j-1$). This phenomenon, which we call mass transfer, could be viewed as incomplete coalescence and is prevented when the daughter distribution function $B$ is of the form 
\begin{align}
B_{i,j}^s = \textbf{1}_{[s, +\infty)} (i) b_{s,i;j} + \textbf{1}_{[s, +\infty)} (j) b_{s,j;i} \label{NMT}
\end{align}
for $i,j\geq 1$ and $s\in \{1,2,\cdots,i+j-1\},$ where $\textbf{1}_{[s, +\infty)}$ denotes the characteristic function of the interval $[s,+\infty)$. In this scenario, when two particles collide, one of them undergoes fragmentation into smaller pieces, without any mass transfer from the other particle.
The breakup kernel $b_{s,i;j}$ gives the rate at which particles with size $s$ are formed due to the collision of particles with sizes $i$ and $j$. It is evident that such a daughter distribution function satisfies \eqref{LMC} if, for $(i,j)\in \mathbb{N}^2$,
\begin{equation*}
\sum_{s=1}^{i-1} sb_{s,i;j} =i, \quad i\ge 2, \qquad b_{1,1;j} = 1.
\end{equation*}

The well-posedness of \eqref{NLDCBE}--\eqref{NLDCBEIC}, under the condition~\eqref{NMT}, is shown in \cite{MAP 23} for a wide range of collision and breakup kernels. In that case, it is also proved in \cite{MAP 23,Laurencot 2001I, SRC 1978} that only monomers remain in the long term, as expected for a fragmentation process. This simple large time behavior might be altered when mass transfer comes into play; that is, when $B$ satisfies~\eqref{LMC} instead of~\eqref{NMT}. The purpose of this work is thus to contribute to a better understanding of the dynamics of \eqref{NLDCBE}-\eqref{NLDCBEIC} with mass transfer and in particular to figure out whether non-trivial stationary solutions may exist. Actually, besides establishing the well-posedness of \eqref{NLDCBE}-\eqref{NLDCBEIC} when $\{B_{j,k}^i\ :\ 1 \le i \le j+k-1\}$ is not concentrated at sizes close to $j+k-1$, we also show that all algebraic superlinear moments become instantaneously finite and stay bounded for positive times. We also identify a class of kinetic coefficients for which there exist non-trivial stationary solutions $w^\star$ to \eqref{NLDCBE}; that is, $w^\star\ne C (\delta_{i,1})_{i\ge 1}$. The proof of this result is based on a dynamical approach, which essentially amounts to identify an appropriate functional framework in which the initial value problem \eqref{NLDCBE}--\eqref{NLDCBEIC} is well-posed, along with a compact, closed, and convex set $\mathcal{Z}$ which is positively invariant for the dynamical system corresponding to \eqref{NLDCBE}--\eqref{NLDCBEIC} (in the sense that $w(t) \in \mathcal{Z}$ for all $t>0$ as soon as $w^{\rm{in}}\in\mathcal{Z}$). If a fixed point theorem can be applied in this functional setting, then, as per a classical argument, at least one stationary solution is guaranteed to exist \cite{AMANN 90, GPV 2004}.

In this paper, we aim to broaden the applicability of the existence results to \eqref{NLDCBE}--\eqref{NLDCBEIC} for the collision kernels
\begin{subequations}\label{AGROWTH}
\begin{equation}\label{AGROWTHa}
 a_{i,j} = A (i^{\alpha}j^{\beta} + i^{\beta} j^{\alpha}), \hspace{.2cm} i,j\geq 1, \hspace{.2cm}\text{where} \hspace{.2cm} \alpha \le \beta \le 1  \hspace{.2cm} \text{and}~~~~ A>0 .
 \end{equation}
 We further assume that 
 \begin{equation}\label{AGROWTHb}
 \beta>0, ~~~~~\alpha<1.
 \end{equation}
\end{subequations}

As for the fragment diffusion function, in addition  to \eqref{LMC}, we assume that, for each $m>1$, there are constants $\epsilon_m\in(0,1)$ and $\kappa_m\ge 1$ such that
\begin{subequations}\label{BCond}
\begin{equation}
\sum_{i=1}^{j+k-1} i^m B_{j,k}^i \leq (1-\epsilon_m) (j^m +k^m) +\kappa_m(jk^{m-1} + j^{m-1}k).  \label{BConda}
\end{equation}
Also, we will put the restriction that, in each collision event, at least two particles are produced, i.e.,
\begin{equation}\label{BCondb}
\sum_{i=1}^{j+k-1}  B_{j,k}^i \ge 2 \qquad j,k\ge 1. 
\end{equation}
\end{subequations}

\begin{example}\label{ex1.1}
It is worth to mention at this point that \eqref{BCond} is satisfied in particular by the uniform daughter distribution function 
\begin{equation}
B_{j,k}^i = \frac{2}{j+k-1}, \quad 1 \le i \le j+k-1,\qquad j,k\geq1. \label{UDDF}
\end{equation} 
Indeed, thanks to the monotonicity and convexity of $x\mapsto x^m$, it follows from \cite[Lemma~7.4.1 and Lemma~7.4.2]{BLL 2019} that
\begin{align*}
\frac{2}{j+k-1}\sum_{i=1}^{j+k-1} i^m & \le \frac{2}{j+k-1}\sum_{i=1}^{j+k-1} \int_i^{i+1} x^m dx \le \frac{2}{j+k-1}\int_1^{j+k} x^m dx \\
& \le \frac{2}{m+1}  \frac{(j+k)^{m+1}-1}{j+k-1} = \frac{2}{m+1} \frac{(j+k)^m(j+k-1) + (j+k)^m - 1}{j+k-1} \\
& \le \frac{2}{m+1} \left[ (j+k)^m + \frac{m(j+k)^{m-1} (j+k-1)}{j+k-1} \right] \\
& \le \frac{2}{m+1} \big( j^m + k^m \big) + \frac{2}{m+1} \left[ (j+k)^m - j^m - k^m + m(j+k)^{m-1} \right] \\
& \le \frac{2}{m+1} \big( j^m + k^m \big) + \frac{2C(m)}{m+1} \big( j^{m-1} k + j k^{m-1} \big) \\
& \qquad + \frac{2^{1+(m-2)_+}}{m+1} \big( j^{m-1} + k^{m-1} \big) \\
& \le \frac{2}{m+1} \big( j^m + k^m \big) + C(m) \big( j^{m-1} k + j k^{m-1} \big),
\end{align*}
which gives~\eqref{BConda} with $\epsilon_m =\frac{m-1}{m+1}\in (0,1)$. The condition~\eqref{BCondb} is obviously satisfied since
\begin{equation*}
	\sum_{i=1}^{j+k-1} B_{j,k}^i = 2.
\end{equation*}
\end{example}

\begin{example}\label{ex1.2}
	We next consider the situation where the collision of two particles results in the complete disintegration of one of the incoming particles into monomers, one of these monomers then attaching to the other particle. In that case, the daughter distribution function $B$ is given by
	\begin{equation}
		\begin{split}
		& B_{1,1}^1 = 2, \qquad B_{1,j}^i = \delta_{i,1} + \delta_{i,j}, \quad 1 \le i \le j, \quad j\ge 2, \\
		& B_{j,k}^i = \frac{j+k-2}{2} \delta_{i,1} + \frac{1}{2} \left( \delta_{i,j+1} + \delta_{i,k+1} \right), \qquad 1 \le i \le j+k-1, \quad j,k\ge 2.
		\end{split} \label{Z6}
	\end{equation}
	Then, for $m>1$, it follows from the convexity of $x\mapsto x^m$ and Young's inequality that, for $j,k\ge 2$,
	\begin{align*}
		\sum_{i=1}^{j+k-1} i^m B_{j,k}^i & = \frac{j+k-2}{2} + \frac{(j+1)^m}{2} + \frac{(k+1)^m}{2} \\
		& \le \frac{j^m}{2m} + \frac{k^m}{2m} + \frac{m-1}{m} - 1 + \frac{j^m + m(j+1)^{m-1}}{2} + \frac{k^m + m(k+1)^{m-1}}{2} \\
		& \le \frac{m+1}{2m} \big( j^m + k^m \big) + m 2^{m-2} \big( j^{m-1} + k^{m-1} \big) \\
		& \le \frac{m+1}{2m} \big( j^m + k^m \big) + m 2^{m-2} \big( j^{m-1} k + j k^{m-1} \big),
	\end{align*}
	so that~\eqref{BConda} is satisfied with $\epsilon_m = \frac{m-1}{2m}\in (0,1)$ and $\kappa_m = m 2^{m-2}$. Also, \eqref{BCondb} is satisfied. Note that the fragment distribution function~\eqref{Z6} is unbounded, unlike the uniform one~\eqref{UDDF}.
\end{example}

\begin{remark}
	In contrast, the daughter distribution function
	\begin{equation*}
		B_{j,k}^i = \delta_{i,1} + \delta_{i,j+k-1}
	\end{equation*}
	does not satisfy~\eqref{BConda}, since
	\begin{equation*}
		\sum_{i=1}^{j+k-1} i^m B_{j,k}^i = 1 + (j+k-1)^m, \qquad j,k\ge 1, \quad m>1.
	\end{equation*}
\end{remark}

We expect the density $\sum_{i=1}^{\infty} i w_i$ to be conserved because particles are neither generated nor destroyed in the interactions represented by \eqref{NLDCBE}. This is mathematically equivalent to
\begin{align}
 \sum_{i=1}^{\infty} iw_i(t) = \sum_{i=1}^{\infty} iw_i^{\rm{in}}, \qquad t\ge 0.\label{MCC}
\end{align}
In other words, the density of the solution $w$ remains constant for all times $t>0$.

\subsection{Outline}
Let us now describe the contents of the paper: Section~\ref{SEC2} introduces the definition of the solution to the discrete collision-induced breakage equations  together with the statements of our main results. In Section~\ref{SEC3}, we provide finite dimensional systems of ordinary differential equations which approximate \eqref{NLDCBE}--\eqref{NLDCBEIC}, as well as moments estimates for their solutions and the proof of the existence of mass-conserving solutions to \eqref{NLDCBE}--\eqref{NLDCBEIC} (Theorem~\ref{TH1}). Section~\ref{SEC4} discusses the continuous dependence of solutions on the initial data, a partial uniqueness result, and the proof of Theorem~\ref{TH2}. Moving on to Section~\ref{SEC5}, we focus on the existence of non-trivial stationary solutions to \eqref{NLDCBE} (Theorem~\ref{TH3}) by using a combination of a dynamical approach and a compactness method. %Finally, Section \ref{SEC6} presents a conjecture on the existence of solutions for a particular class of collisional kernels, with the help of numerical simulations.

\section{Main results}\label{SEC2}

\subsection{Notation}
For $\gamma_0 \geq 0$, let $Y_{\gamma_0}$ be the Banach space defined by
\begin{align*}
Y_{\gamma_0} = \Big\{ y =(y_i)_{i\in\mathbb{N}}: y_i \in \mathbb{R}, \sum_{i=1}^{\infty} i^{\gamma_0} |y_i| < \infty \Big\}
\end{align*}
with the norm 
\begin{align*}
\|y\|_{\gamma_0} =\sum_{i=1}^{\infty} i^{\gamma_0} |y_i|, \qquad y\in Y_{\gamma_0}.
\end{align*}

We will use the positive cone $Y_{\gamma_0}^+$ of $Y_{\gamma_0}$, that is,
\begin{align*}
Y_{\gamma_0}^+ =\{y\in Y_{\gamma_0}: ~~y_i \geq 0~~\text{for each}~~ i\geq 1\},
\end{align*}
and we denote the moment of order $\gamma_0$ by 
\begin{align*}
M_{\gamma_0}(w) =\sum_{i=1}^{\infty} i^{\gamma_0} w_i, \qquad w\in Y_{\gamma_0}^+.
\end{align*}

It is worth noting that the quantity $M_0(w)$ of a particular cluster distribution $w$ represents the total number of clusters present in the system, and the quantity $M_1(w)$ estimates the overall density or mass of the cluster distribution $w$.

\subsection{Weak solutions}

Let us now define what we mean by a solution to \eqref{NLDCBE}--\eqref{NLDCBEIC}.

\begin{definition} \label{DEF1}
Let $T\in(0,+\infty]$ and $w^{\rm{in}}= (w_{i}^{\rm{in}})_{i \geq 1}\in Y_1^+$ be a sequence of non-negative real numbers. A solution $w=(w_i)_{i \geq 1} $ to \eqref{NLDCBE}--\eqref{NLDCBEIC} on $[0,T)$ is a sequence of non-negative continuous functions satisfying for each $i\geq 1$ and $t\in(0,T)$ 
\begin{enumerate}
\item $w_i\in \mathcal{C}([0,t))$, $\sum_{j=1}^{\infty}  a_{i,j} w_j \in L^1(0,t)$, $ \sum_{j=i+1}^{\infty}\sum_{k=1}^{j-1}B_{j-k,k}^i a_{j-k,k} w_{j-k} w_k \in L^1(0,t)$,
\item and there holds 
\begin{align}
w_i(t) = w_{i}^{\rm{in}} + \int_0^t \Big(\frac{1}{2} \sum_{j=i+1}^{\infty} \sum_{k=1}^{j-1} B_{j-k,k}^i a_{j-k,k} w_{j-k}(\tau) w_k(\tau) -\sum_{j=1}^{\infty} a_{i,j} w_i(\tau) w_j(\tau) \Big) d\tau. \label{IVOE}
\end{align}
\end{enumerate}
The solution $w$ is said to be mass-conserving on $[0,T)$ if 
\begin{align}
M_1(w(t)) = M_1(w^{\rm{in}}),\qquad t\in[0,T). \label{MC}
\end{align}
\end{definition}

\subsection{Well-posedness}

We first deal with the existence of global mass-conserving solutions to \eqref{NLDCBE}--\eqref{NLDCBEIC} when the kinetic coefficients satisfy \eqref{LMC}, \eqref{AGROWTH} and \eqref{BCond}.

\begin{theorem} \label{TH1}
Assume that the kinetic coefficients $(a_{i,j})$ and $(B_{j,k}^i)$ satisfy \eqref{LMC}, \eqref{AGROWTH} and \eqref{BCond}. Given $w^{\rm{in}} \in Y_1^+$, there is at least one mass-conserving solution $w$ to \eqref{NLDCBE}--\eqref{NLDCBEIC} on $[0,+\infty)$ such that
\begin{equation}
M_0(w(t) \ge M_0(w^{\rm{in}}), \qquad t\ge 0, \label{ZEROMOM}
\end{equation}
and, for any $m>1$ and $t>0$, $w(t)\in Y_m$ with
\begin{equation}
M_m(w(t)) \le F_m\Big(1+ \frac{1}{t}\Big)^{\frac{m-1}{\beta}},\label{MMOMTIM}
\end{equation}
where $F_m$ is a positive constant depending only on $A$, $\alpha$, $\beta$, in \eqref{AGROWTHa}, $\epsilon_m$, $\kappa_m$ in \eqref{BConda}, $M_1(w^{\rm{in}})$ and $M_0(w^{\rm{in}})$.\\
In addition, for every $m>1$ and $\rho_1 \ge \rho_0 >0$, there is $\mu_m>0$ depending only on $A$, $\alpha$, $\beta$, $\epsilon_m$, $\kappa_m$, $\rho_1$ and $\rho_0$ such that if $w^{\rm{in}} \in Y_m$ , with $M_1(w^{\rm{in}}) \le \rho_1$ and $M_0(w^{\rm{in}}) \ge \rho_0$, then 
\begin{equation}
M_m(w(t)) \le \max\big\{M_m(w^{\rm{in}}), \mu_m\big\}, \qquad t>0.\label{MMOMTIMI}
\end{equation}
\end{theorem}

Besides settling the existence issue for \eqref{NLDCBE}--\eqref{NLDCBEIC}, Theorem~\ref{TH1} also shows that any algebraic superlinear moment of the constructed solution is finite for positive times. Furthermore, it follows from~\eqref{MMOMTIMI} that some balls in $Y_m$ are positively invariant for the dynamics of \eqref{NLDCBE}--\eqref{NLDCBEIC}.
	
\begin{remark}
	Unfortunately, the classical multiplicative collision kernel $a_{i,j}=Aij$ (corresponding to $\alpha=\beta=1$) is excluded from our analysis due to the assumption~\eqref{AGROWTHb}. It is yet unclear whether Theorem~\ref{TH1} extends to that case as well.
\end{remark}

We next supplement Theorem~\ref{TH1} with a uniqueness result for initial conditions in a smaller space.

\begin{theorem} \label{TH2}
Assume that the kinetic coefficients $(a_{i,j})$ and $(B_{j,k}^i)$ satisfy \eqref{LMC}, \eqref{AGROWTH} and \eqref{BCond}. Given $w^{\rm{in}} \in Y_1^+ \cap Y_{1+\beta}$, there is a unique mass-conserving solution $w=\Psi(\cdot;w^{\rm{in}})$ to \eqref{NLDCBE}--\eqref{NLDCBEIC} on $[0,+\infty)$ such that $w\in L^{\infty}((0,+\infty), Y_{1+\beta}) $ and $w$ satisfies \eqref{ZEROMOM}, \eqref{MMOMTIM} and \eqref{MMOMTIMI}. In addition, given $0<\rho_0 \le \rho_1$ and $R\ge\mu_{1+\beta}$, $\Psi(\cdot;w^{\rm{in}})$ is a dynamical system in 
\begin{align}
\mathcal{S}(\rho_0, \rho_1,R) := \Big\{y \in Y_1^+: \rho_0 \le M_0(y) \le M_1(y)= \rho_1 ~~~~\text{and}~~~~M_{1+\beta}(y) \le R \Big\} \label{SPACE}
\end{align} 
 for the topology induced by the norm of $Y_1$.
 \end{theorem}

\subsection{Stationary solutions}

The last result deals with the existence of non-trivial stationary solutions to \eqref{NLDCBE}. Let us first point out that, given $\rho_1>0$, the sequence $\bar{w}_{\rho_1} := \rho_1 (\delta_{i,1})_{i\ge 1}$ is a stationary solution to~\eqref{NLDCBE} with total mass $M_1(\bar{w}_{\rho_1}) = \rho_1$ and it follows from~\cite{MAP 23, Laurencot 2001I} that it is the only one with that mass when $B$ satisfies additionally~\eqref{NMT}. It is thus of interest to figure out whether mass transfer gives rise to other stationary solutions and a contribution in that direction is reported in the next result.

\begin{theorem} \label{TH3}
Assume that the kinetic coefficients $(a_{i,j})$ and $(B_{j,k}^i)$ satisfy \eqref{LMC}, \eqref{AGROWTH} and \eqref{BCond}. 
\begin{itemize}
\item[(a)] Assume further that
\begin{equation}
	\sum_{i=1}^{j+k-1} B_{j,k}^i = 2, \qquad j,k \ge 1. \label{Z1}
\end{equation}
Then, given $\rho_1 > \rho_0 > 0$, there is at least one stationary solution $w^{\star} \in Y_1^+$ to \eqref{NLDCBE} satisfying 
\begin{equation}
	M_1(w^{\star}) = \rho_1 > \rho_0 = M_0(w^{\star}) \label{Z2}
\end{equation} 
and $w^{\star}\in \bigcap_{m\ge 1} Y_m$.
\item[(b)] Assume further that
\begin{equation}
\sum_{i=2}^j B_{1,j}^i \ge 1, \quad \sum_{i=2}^{j+k-1} B_{j,k}^i \ge 2, \qquad j,k\ge 2. \label{Z3}
\end{equation}
Then, given $\rho_1 > \rho_0 > 0$, there is at least one stationary solution $w^{\star} \in Y_1^+$ to \eqref{NLDCBE} satisfying $w^{\star}\in \bigcap_{m\ge 1} Y_m$ and
\begin{equation}
	M_1(w^\star) = \rho_1, \quad M_0(w^\star)\ge \rho_0 \;\;\text{ and }\;\; w_1^\star <  M_0(w^\star). \label{Z4}
\end{equation}	
\end{itemize}
\end{theorem}

Owing to~\eqref{Z2} and~\eqref{Z4}, the stationary solution $w^\star$ constructed in~Theorem~\ref{TH3} differs from $\bar{w}_{\rho_1}$, so that we have constructed non-trivial stationary solutions to~\eqref{NLDCBE}. Furthermore, in case~(a), it follows from~\eqref{Z2} that there is at least infinitely many different stationary solutions with a given total mass $\rho_1$. This is in particular the case for the uniform daughter distribution function given by~\eqref{UDDF} which obviously satisfies~\eqref{Z1}. An interesting issue is then whether~\eqref{Z2} determines uniquely the corresponding stationary solution and whether the latter attracts the dynamics. Finally, while Theorem~\ref{TH3}~(b) does not apply to the fragment distribution function defined by~\eqref{Z6} as it fails to satisfy~\eqref{Z4}, it applies to the daughter distribution function
\begin{equation*}
	\begin{split}
		& B_{1,1}^1 = 2, \qquad B_{1,j}^i = \delta_{i,1} + \delta_{i,j}, \quad 1 \le i \le j, \quad j\ge 2, \\
		& B_{2,2}^2 = 2, \qquad B_{2,j}^i = \delta_{i,2} + \delta_{i,j}, \quad 1 \le i \le j+1, \quad j\ge 3, \\
		& B_{j,k}^i = \frac{j+k-6}{2} \delta_{i,1} + \delta_{i,2} + \frac{1}{2} \big( \delta_{i,j+1} + \delta_{i,k+1} \big), \qquad 1 \le i \le j+k-1, \quad j,k\ge 3.
	\end{split} 
	\end{equation*}
For such a daughter distribution function, the collision of two clusters with respective sizes $j\ge 2$ and $k\ge 2$ always produces a cluster with size~$2$ and a cluster with size larger or equal to $2$, so that~\eqref{Z4} is satisfied. A similar argument as in Example~\ref{ex1.2} shows that it satisfies~\eqref{BCond}.

\section{Existence} \label{SEC3}
In this section, we fix kinetic coefficients $(a_{i,j})$ and $(B_{j,k}^i)$ satisfying  \eqref{LMC}, \eqref{AGROWTH}, \eqref{BCond} and $w^{\rm{in}} \in Y_1^+$, $w^{\rm{in}}\not\equiv 0$. We also fix $\rho_1 \ge \rho_0>0$ such that 
\begin{equation}
M_1(w^{\rm{in}}) \le \rho_1 \qquad \text{and} \qquad M_0(w^{\rm{in}}) \ge \rho_0. \label{M10BND}
\end{equation}

\subsection{Approximation}

For $l \ge 1$ , we define a sequence of approximations of $w^{in}$ and $a_{i,j}$ by 
\begin{equation}
w^{\rm{in},l} = w^{\rm{in}} \textbf{1}_{[1,l]}, \label{Z5}
\end{equation}
and
\begin{equation}
a_{i,j}^l = A \big(\min\{i,l\}^{\alpha_+} i^{\alpha-\alpha_+} j^{\beta} + i^{\beta} \min\{j,l\}^{\alpha_+} j^{\alpha-\alpha_+}\big)\le Al (i+j), \qquad i,j\ge 1, \label{ASSUM3}
\end{equation}
where $\alpha_+ = \max\{ \alpha , 0\}$.

Owing to~\eqref{ASSUM3}, the collision kernel $(a_{i,j}^l)$ has linear growth. We may then use \cite[Theorem~3.1 and Proposition~3.7]{Laurencot 2001I} to establish the existence of a solution $w^l$ to
\begin{align}
\frac{dw_i^l}{dt}  =&\frac{1}{2} \sum_{j=i+1}^{\infty} \sum_{k=1}^{j-1} B_{j-k,k}^i a_{j-k,k}^l w_{j-k}^l w_k^l -\sum_{j=1}^{\infty} a_{i,j}^l w_i^l w_j^l,  \hspace{.5cm} i \in \mathbb{N}, \label{TNLDCBE}\\
w_i^l(0) &= w_i^{\rm{in},l}, \hspace{.5cm} i \in \mathbb{N}.\label{TNLDCBEIC}
\end{align}

More precisely we have the following result.

\begin{prop}\label{prop0}
There is at least one  mass-conserving solution $w^l$ to \eqref{TNLDCBE}--\eqref{TNLDCBEIC} on $[0,+\infty)$. Moreover, $w^l$ belongs to $L_{loc}^{\infty}([0,+\infty), Y_m)$ for all $m>1$.
\end{prop}

We now report a classical identity for $w^l$. It is usually valid for bounded sequences but the summability properties of $w^l$ stated in Proposition~\ref{prop0} allows us to handle any sequence with algebraic growth.

\begin{lemma} \label{MESTELMMA}
Let $(\mu_i)_{i\ge 1}$ be a non-negative sequence such that $(i^{-m}\mu_i )$ is bounded for some $m\ge 1$. Then there holds
\begin{align}
\frac{d}{dt}\sum_{i=1}^{\infty} \mu_i w_i^l = \frac{1}{2} \sum_{j=1}^{\infty} \sum_{k=1}^{\infty} \Big(\sum_{i=1}^{j+k-1} \mu_i B_{j,k}^i -\mu_j-\mu_k\Big) a_{j,k}^l w_j^l w_k^l.\label{GME}
\end{align}
\end{lemma}

\subsection{Moment Estimates}

We provide several  estimates for superlinear moments of solutions to  \eqref{TNLDCBE}--\eqref{TNLDCBEIC}. These bounds will be crucial for the proofs of our main results.

Before moving further, let us prove the following lemma which will be helpful on getting lower bounds.

\begin{lemma} \label{ZEROBNDLEM}
Let $w^l$ be the solution to \eqref{TNLDCBE}--\eqref{TNLDCBEIC} given by Proposition~\ref{prop0}. There are $l_0 \ge 1$  depending only on $w^{\rm{in}}$, and $\delta_0>0$ depending only on $\rho_0$, $\rho_1$ and $\alpha$ such that, for $t\ge 0$ and $l \ge l_0$,

\begin{align}
	M_1(w^l(t)) & = M_1(w^{\rm{in},l}) \le\rho_1 , \label{LEMEST1}\\
	M_0(w^l(t)) & \ge M_0(w^{\rm{in},l}) \ge \delta_0, \label{LEMEST2}\\
	M_{\alpha}(w^l(t)) & \ge \delta_0>0. \label{LEMEST3}
\end{align}
\end{lemma}

\begin{proof}

First, the conservation of matter and the upper bound \eqref{LEMEST1} readily follow from \eqref{M10BND}, \eqref{Z5}, \eqref{ASSUM3} and Proposition~\ref{prop0}. Next, we pick $l_0\ge 1$ large enough such that
\begin{equation}
\sum_{i=l_0+1}^{\infty} w_i^{\rm{in}} \le \frac{M_0(w^{\rm{in}})}{2}. \label{TAILBND}
\end{equation} 
From \eqref{BCondb} and Lemma~\ref{MBLemma} (with $\mu_i=1$), we infer that 
\begin{equation*}
\frac{d}{dt} M_0(w^l) =\frac{1}{2}\sum_{j=1}^{\infty} \sum_{k=1}^{\infty} \Big( \sum_{i=1}^{j+k-1} B_{j,k}^i -2 \Big)a_{j,k}^l w_j^l w_k^l \ge 0, 
\end{equation*}
which further gives, together with \eqref{M10BND} and \eqref{TAILBND}, 
\begin{equation}
M_0(w^l(t))\ge M_0(w^{\rm{in},l})\ge \frac{M_0(w^{\rm{in}})}{2}\ge \frac{\rho_0}{2}, \qquad  \text{for}~~~~l\ge l_0. \label{Z7}
\end{equation}
 Now, either $\alpha\in [0,1)$ and~\eqref{Z7} gives
\begin{equation*}
	M_{\alpha}(w^l(t)) \ge M_0(w^l(t)) \ge \frac{\rho_0}{2};
\end{equation*}
or $\alpha<0$ and we deduce from~\eqref{LEMEST1}, \eqref{Z7} and H\"older's inequality that
\begin{align*}
	\frac{\rho_0}{2}\le M_0(w^l(t)) \le M_{\alpha}(w^l(t))^{\frac{1}{1-\alpha}} M_{1}(w^l(t))^{\frac{-\alpha}{1-\alpha}}\le M_{\alpha}(w^l(t))^{\frac{1}{1-\alpha}} \rho_1^{\frac{|\alpha|}{1-\alpha}},
\end{align*}
which implies that
\begin{equation*}
	\Big(\frac{\rho_0}{2}\Big)^{1-\alpha} \rho_1^{\alpha} \le M_{\alpha}(w^l(t)).
\end{equation*}
Setting $\delta_0 = \min\Big\{\frac{\rho_0}{2}, \Big(\frac{\rho_0}{2}\Big)^{1-\alpha} \rho_1^{\alpha}\Big\}$ completes the proof.
\end{proof}

Based on these lower bounds, we exploit the inequality~\eqref{BCond} to derive uniform estimates on the superlinear  moments of $w^l$.

\begin{lemma} \label{MBLemma}
Let $l\ge l_0$ and $w^l$ be the solution to \eqref{TNLDCBE}--\eqref{TNLDCBEIC} given by Proposition~\ref{prop0}. Given $m > 1$, there is a positive constant $F_m$ depending only on $(A, \alpha, \beta)$ in \eqref{AGROWTHa}, $(\epsilon_m,\kappa_m)$ in \eqref{BConda},  $(\rho_0,\rho_1)$ in \eqref{M10BND} and  $m$ such that
	 \begin{equation}
	 M_m(w^l(t)) \le F_m \Big(1+\frac{1}{t}\Big)^{\frac{m-1}{\beta}}, \qquad t>0. \label{FOM}
	 \end{equation}
	 In addition, there is a positive constant $\mu_m$ depending only on $A$, $\alpha$, $\beta$, $\epsilon_m$, $\kappa_m$, $\rho_0$, $\rho_1$ and  $m$ such that  
	 \begin{equation}
		M_m(w^l(t)) \le \max\Bigg\{ M_m(w^{\rm{in},l}), \mu_m \Bigg\}, \qquad t>0. \label{FOM1}
	\end{equation}
	\end{lemma}
	
\begin{proof}
Fix $m>1$. We take $\mu_i=i^m$ in \eqref{GME} to get 
 \begin{align*}
 \frac{d}{dt} \sum_{i=1}^{\infty} i^m w_i^l= \sum_{j=1}^{\infty} \sum_{k=1}^{\infty} \Big[\sum_{i=1}^{j+k-1} i^m B_{j,k}^i - j^m - k^m\Big] a_{j,k}^l w_j^l w_k^l.
 \end{align*}
 Now using  \eqref{BConda}, we arrive at 
 \begin{align}
 \frac{d}{dt}M_{m}(w^l) \le& \frac{1}{2} \sum_{j=1}^{\infty} \sum_{k=1}^{\infty} \Big[\kappa_m (jk^{m-1} + j^{m-1} k) -\epsilon_m(j^m + k^m) \Big] a_{j,k}^l w_j^l w_k^l \nonumber\\
 & \le \kappa_m \sum_{j=1}^{\infty} \sum_{k=1}^{\infty}  jk^{m-1} a_{j,k}^l w_j^l w_k^l -\epsilon_m \sum_{j=1}^{\infty} \sum_{k=1}^{\infty}  j^m a_{j,k}^l w_j^l w_k^l. \label{EST01}
 \end{align}
 Let us first look at the dissipative term  in \eqref{EST01}, 
 \begin{align}
 \sum_{j=1}^{\infty}\sum_{k=1}^{\infty} j^m a_{j,k}^l w_j^l w_k^l &=\sum_{j=1}^{\infty} \sum_{k=1}^{\infty} \big[j^m \min\{j,l\}^{\alpha_+} j^{\alpha-\alpha_+} k^{\beta} + j^{m+\beta} \min\{k,l\}^{\alpha_+} k^{\alpha-\alpha_+} \big]w_j^l w_k^l \nonumber\\
 &\ge \Big( \sum_{j=1}^{\infty}j^{m+\beta} w_j^l\Big)\Big(\sum_{k=1}^{\infty} \min\{k,l\}^{\alpha_+} k^{\alpha-\alpha_+} w_k^l\Big) \nonumber\\
 & \ge \Big( \sum_{j=1}^{\infty}j^{m+\beta} w_j^l\Big) \Big(\sum_{k=1}^{\infty} k^{\alpha} w_k^l\Big). \label{EST02}
 \end{align}
 Since
 \begin{equation*}
 \sum_{k=1}^{\infty} k^{\alpha} w_k^l =M_\alpha(w^l) \ge \delta_0
 \end{equation*}
 by  Lemma~\ref{ZEROBNDLEM}, we deduce that 
 \begin{equation}
\sum_{j=1}^{\infty} \sum_{k=1}^{\infty} j^m a_{j,k}^l w_j^l w_k^l \ge \delta_0 M_{m+\beta}(w^l). \label{EST02a}
 \end{equation}

Next, considering the first term on the right-hand side of \eqref{EST01}, it follows from \eqref{ASSUM3} that
 \begin{align}
\sum_{j=1}^{\infty} \sum_{k=1}^{\infty} j k^{m-1} a_{j,k}^l w_j^l w_k^l=& \sum_{j=1}^{\infty}\sum_{k=1}^{\infty} j^{1+\alpha-\alpha_+} \min\{j,l\}^{\alpha_+} k^{m+\beta-1} \\
& \qquad + \sum_{j=1}^{\infty}\sum_{k=1}^{\infty} j^{1+\beta}k^{m-1+\alpha-\alpha_+} \min\{k,l\}^{\alpha_+}\big]w_j^l w_k^l \nonumber\\
 &\le M_{1+\alpha_+}(w^l) M_{m+\beta-1}(w^l)+ M_{1+\beta}(w^l) M_{m-1+\alpha_+}(w^l). \label{EST03}
 \end{align}
 In order to continue, we estimate the involved moments with the help of~\eqref{AGROWTHa} and H\"older's inequality and find  
 \begin{align*}
 M_{1+\alpha_+}(w^l) \le&(M_1(w^l))^{\frac{m+\beta-1-\alpha_+}{m+\beta-1}}(M_{m+\beta}(w^l))^{\frac{\alpha_+}{m+\beta-1}}, \\
 M_{m+\beta-1}(w^l) \le& (M_{\beta}(w^l))^{\frac{1}{m}} (M_{m+\beta}(w^l))^{\frac{m-1}{m}}\le (M_1(w^l))^{\frac{1}{m}}(M_{m+\beta}(w^l))^{\frac{m-1}{m}}, \\
 M_{1+\beta}(w^l) \le&(M_1(w^l))^{\frac{m-1}{m+\beta-1}}(M_{m+\beta}(w^l))^{\frac{\beta}{m+\beta-1}}, \\
 M_{m+\alpha_+-1}(w^l) \le& (M_{\alpha_+}(w^l))^{\frac{1+\beta-\alpha_+}{m+\beta-\alpha_+}} (M_{m+\beta}(w^l))^{\frac{m-1}{m+\beta-\alpha_+}}\nonumber\\
 &\le (M_1(w^l))^{\frac{1+\beta-\alpha_+}{m+\beta-\alpha_+}})(M_{m+\beta}(w^l))^{\frac{m-1}{m+\beta-\alpha_+}},
 \end{align*}
 so that, by \eqref{LEMEST1}, we get
 \begin{align}
 M_{1+\alpha_+}(w^l) M_{m+\beta-1}(w^l)  \le& C(m) M_{m+\beta}(w^l)^{\theta},\label{EST1}\\
 M_{1+\beta}(w^l) M_{m+\alpha_+-1}(w^l)  \le& C(m) M_{m+\beta}(w^l)^{\omega},\label{EST2}
 \end{align}
 where 
 \begin{align*}
 \theta &= \frac{(m-1)(m-1+\beta)+m\alpha_+}{m(m+\beta-1)}>0,\\
 \omega &=\frac{(m+\beta-1)^2 +\beta(1-\alpha_+)}{m+\beta-1)(m+\beta-\alpha_+)}>0.
 \end{align*}
 
 We note that, since $m>1$, $\beta\ge \alpha_+$ and $\alpha_+<1$ by \eqref{AGROWTHb},
 \begin{align*}
 1-\theta =& \frac{m+\beta-1-m\alpha_+}{m(m+\beta-1)}\ge \frac{m(1-\alpha_+)+ \alpha_+-1}{m(m+\beta-1)}= \frac{(m-1)(1-\alpha_+)}{m(m-\beta+1)}>0,\\
 1-\omega=& \frac{(m-1)(1-\alpha_+)}{(m+\beta-1)(m+\beta-\alpha_+)}>0.
 \end{align*}

 On collecting the estimates from~\eqref{EST02a}, \eqref{EST1}, \eqref{EST2} and inserting into~\eqref{EST01}, we obtain
 \begin{equation*}
 \frac{d}{dt} M_m(w^l)\le  C(m)[(M_{m+\beta}(w^l))^{\theta} +(M_{m+\beta}(w^l))^{\omega}]-\epsilon_m \delta_0 M_{m+\beta}(w^l).
 \end{equation*}
 	Since $(\theta,\omega)\in (0,1)^2$, we deduce from Young's inequality that
 	\begin{equation}
 	\frac{d}{dt} M_m(w^l)\le C(m) -\frac{\epsilon_m \delta_0}{2} M_{m+\beta}(w^l). \label{EST3}
 	\end{equation}
 Since $m \in [1,m+\beta]$, we use  H\"{o}lder's inequality and~\eqref{LEMEST1} to show that
 \begin{equation*}
 M_m(w^l) \le M_1(w^l)^{\frac{\beta}{m+\beta-1}} (M_{m+\beta}(w^l))^{\frac{m-1}{m+\beta-1}} \le \rho_1^{\frac{\beta}{m+\beta-1}} (M_{m+\beta}(w^l))^{\frac{m-1}{m+\beta-1}}.
\end{equation*}  
Hence ,
\begin{equation}
	\rho_1^{-\frac{\beta}{m-1}} (M_m(w^l))^{\frac{m+\beta-1}{m-1}} \le  M_{m+\beta}(w^l). \label{EST8}
\end{equation}
 Plugging~\eqref{EST8} in~\eqref{EST3}, we get
\begin{align}
  \frac{d}{dt}M_m(w^l) +C_{1}(m) (M_{m}(w^l))^{\frac{m+\beta-1}{m-1}}\le C_{2}(m), \qquad t\ge0. \label{EST7}
 \end{align}
 
 Introducing 
 \begin{equation*}
 	X(t) = \Big(X_1 + \frac{X_2}{t}\Big)^{\frac{m-1}{\beta}}, \qquad t>0,
 \end{equation*}
 with 
 \begin{equation*} 
 	X_2 = \frac{m}{\beta C_{1}(m)}, \qquad X_1 = \Bigg( \frac{mC_{2}(m)}{C_{1}(m)} \Bigg)^{\frac{\beta}{m+\beta-1}},
\end{equation*}
an easy computation shows that $X$ is a supersolution to \eqref{EST7} such that $X(t)\to +\infty $ as $t\to 0$. The comparison principle then implies that $M_m(w^l(t))\le X(t)$ for $t >0$, which proves \eqref{FOM} for $m > 1$ with $F_m= \max\big\{X_1, X_2\big\}^{\frac{m-1}{\beta}}.$

 We also observe that~\eqref{EST7} and the comparision principle imply that 
\begin{equation*}
M_m(w^l(t)) \le \max\Bigg\{ M_m(w^{\rm{in},l}), \Bigg(\frac{C_{2}(m)}{C_{1}(m)}\Bigg)^{\frac{m-1}{m-1+\beta}}\Bigg\},
\end{equation*}
for $t\ge 0$, thereby completing the proof of Lemma~\ref{MBLemma}, recalling that $M_m(w^{\rm{in},l})\le M_m(w^{\rm{in}})$ by~\eqref{Z5}.
\end{proof}

Since we aim at showing the existence of a weak solution to \eqref{NLDCBE}--\eqref{NLDCBEIC} for arbitrary $w^{\rm{in}}\in Y_1^+$, the previous study of the time evolution of superlinear moments is not sufficient due to the singularity in \eqref{FOM} as $t\to 0$. To overcome this difficulty, we proceed as in \cite{BLL 2019, Laurencot 2001I} with the help of de la Vall\'ee Poussin theorem. This classical approach requires to introduce some notation. We denote by $\mathcal{G}_1$ the set of non-negative and convex functions $G \in C^1([0,+\infty))\cap W_{\text{loc}}^{2,\infty}(0,+\infty)$ such that $G(0)=0$, $G'(0)\geq 0$, $G'$ is a concave function and, for any $p\in(1,2]$,
\begin{align}
\sigma_p(G) :=\sup_{\zeta\ge 0}\frac{G(\zeta)}{\zeta^p}<+\infty.
\end{align}
 We also denote by $\mathcal{G}_{1, \infty}$  the set of functions $G \in \mathcal{G}_1$ satisfying, in addition,
\begin{align}
\lim_{\zeta \to +\infty} G'(\zeta) = \lim_{\zeta\to +\infty} \frac{G(\zeta)}{\zeta} = +\infty.
\end{align}  
	
An interesting property of functions in $\mathcal{G}_1$ is the following inequality \cite[Proposition~7.1.9]{BLL 2019}.

\begin{lemma}
For $G \in \mathcal{G}_1$ and $i,j\ge 1$ there holds 
\begin{equation}
(i+j)\big(G(i+j)-G(i) -G(j)\big) \le 2 \big(i G(j) +j G(i)\big). \label{GInequality}
\end{equation}
\end{lemma}

We  now show that the generalized moments built upon $G\in \mathcal{G}_1$ propagates  with time.

\begin{lemma} \label{prop1}
Consider $T \in (0,+\infty)$ and $G \in \mathcal{G}_1$. Then there is a positive constant $\Lambda(T)$ depending only on $A$, $\alpha$, $\beta$, $F_\lambda$, $G$ and $T$ such that, for each $l \geq l_0$, there holds 
\begin{equation}
\sum_{i=1}^{\infty} G(i) w_i^l(t) \leq \Lambda(T) \Big(1+\sum_{i=1}^{\infty} G(i) w_i^{\rm{in},l}\Big), \qquad t \in[0,T], \label{PROPEQN1}
\end{equation}

\begin{align}
0 & \leq \int_0^T\sum_{i=1}^{\infty} \sum_{j=1}^{\infty} \sum_{s=1}^{i+j-1}\Big(\frac{G(i+j)}{i+j}-\frac{G(s)}{s}\Big)s B_{i,j}^s  a_{i,j}^l w_i^l(t) w_j^l(t) dt \nonumber \\
& \hspace{7cm}\leq \Lambda(T) \Big(1+\sum_{i=1}^l G(i) w_i^{\rm{in}}\Big),\label{PROPEQN2}
\end{align}
where $F_\lambda$ is defined in Lemma~\ref{MBLemma} and $\lambda\in (\beta,1+\beta)$ is defined below in the proof and only depends on $\alpha$ and $\beta$.
\end{lemma}

Recall that the non-negativity of the integral in~\eqref{PROPEQN2} follows from the monotonicity of $\zeta \mapsto \frac{G(\zeta)}{\zeta}$, which is due to the convexity of $G$.

\begin{proof}
For $l \geq l_0$ and $t\ge 0$ we put 
\begin{equation*}
M_G^l(t) =\sum_{i=1}^{\infty} G(i) w_i^l(t).
\end{equation*}
It follows from~\eqref{LMC} and~\eqref{GME} that
\begin{align}
\frac{d}{dt} M_G^l =& \frac{1}{2} \sum_{j=1}^{\infty} \sum_{k=1}^{\infty} \Big( \sum_{i=1}^{j+k-1} G(i) B_{j,k}^i - G(j)-G(k)\Big) a_{j,k}^l w_j^l w_k^l \\
&= \frac{1}{2} \sum_{j=1}^{\infty} \sum_{k=1}^{\infty} \big( G(j+k) - G(j)-G(k)\big) a_{j,k}^l w_j^l w_k^l \nonumber\\
 &\qquad-\sum_{j=1}^{\infty} \sum_{k=1}^{\infty} \Big(G(j+k)-\sum_{i=1}^{j+k-1} G(i) B_{j,k}^i\Big) a_{j,k}^l w_j^l w_k^l. \label{PROPEQN3}
\end{align}
On the one hand, since $G$ is convex, the function $\zeta \mapsto \frac{G(\zeta)}{\zeta}$ is non-decreasing on $(0,+\infty)$ and it follows from \eqref{LMC} that
\begin{align}
G(j+k) - \sum_{i=1}^{j+k-1} G(i) B_{j,k}^i & = \frac{G(j+k)}{j+k} \sum_{i=1}^{j+k-1} i B_{j,k}^i - \sum_{i=1}^{j+k-1} \frac{G(i)}{i} i B_{j,k}^i \\
& = \sum_{i=1}^{j+k-1} \left( \frac{G(j+k)}{j+k} - \frac{G(i)}{i} \right) i B_{j,k}^i\ge 0. \label{PROPEQN4}
\end{align}
On the other hand, by~\eqref{AGROWTHa} and~\eqref{GInequality},
\begin{align}
\sum_{j=1}^{\infty} \sum_{k=1}^{\infty} \Big( G(j+k) - G(j)-&G(k)\Big) a_{j,k}^l w_j^l w_k^l \nonumber \\
&\le 2 \sum_{j=1}^{\infty} \sum_{k=1}^{\infty}\frac{jG(k) +kG(j)}{j+k}a_{j,k}^l w_j^l w_k^l \nonumber\\
& \le 4A \sum_{j=1}^{\infty} \sum_{k=1}^{\infty}\frac{jG(k) +kG(j)}{j+k} j^{\alpha} k^{\beta} w_j^l w_k^l.
\end{align}
For some yet undetermined $p\in (1,2]$ and $(\theta_1,\theta_2)\in(0,1)^2$, we compute
\begin{align}
 \sum_{j=1}^{\infty} \sum_{k=1}^{\infty} &\frac{jG(k) +kG(j)}{j+k} j^{\alpha}k^{\beta} w_j^l w_k^l  \nonumber\\
 & \le \sigma_p (G) \sum_{j=1}^{\infty} \sum_{k=1}^{\infty} \frac{(j^{1+\alpha}k^{p+\beta} +k^{1+\beta}j^{p+\alpha})}{j+k}  w_j^l w_k^l\nonumber\\
 & = \sigma_p (G) \sum_{j=1}^{\infty} \sum_{k=1}^{\infty} \left[  \frac{j^{1+\alpha}k^{p+\beta}}{(j+k)^{\theta_1}(j+k)^{1-\theta_1}} + \frac{k^{1+\beta}j^{p+\alpha}}{(j+k)^{\theta_2}(j+k)^{1-\theta_2}} \right] w_j^l w_k^l \nonumber\\
& \leq \sigma_p(G) \big( M_{\alpha+1-\theta_1}(w^l) M_{p+\beta+\theta_1-1}(w^l) + M_{\beta+1-\theta_2}(w^l) M_{\alpha+p+\theta_2-1}(w^l)\big). \label{PROPEQN5}
\end{align}
Now two cases arise depending on the value of $\beta$.
\begin{itemize}
\item[\textbf{Case 1}:]$\beta<1$. Let us choose 
\begin{align*}
\theta_1 = \theta_2 = \hat{\theta}\in \Big(0,\frac{1-\beta}{2} \Big]\cap \Big(0,\frac{1}{2}\Big),~~~ p =1+\hat{\theta}\in(1,2).
\end{align*}
Then,
\begin{align*}
\alpha+1-\hat{\theta}\le &\beta +1 -\hat{\theta} < 1+\beta,\\
p+\alpha+\hat{\theta}-1 &\le p+\beta +\hat{\theta} -1=\beta+ 2 \hat{\theta}\le \beta +1-\beta =1.
\end{align*}
\item[\textbf{Case 2:}]$\beta=1$. Then $\alpha<1$ by \eqref{BCondb} and we choose 
\begin{align*}
\theta_1 =\alpha,~~~ p\in(1,2-\alpha),~~~\theta_2 \in(0,2-p-\alpha).
\end{align*}
Then, 
\begin{align*}
\alpha+1 -\theta_1 =1,~~~&p+\beta+\theta_1-1= p+\alpha \le 2-\theta_2 <2 =1+\beta,\\
\beta+1-\theta_2= 2-\theta_2<2, &~~~p+\alpha+\theta_2 -1 \le p+\alpha-1+2-p -\alpha=1
\end{align*}
\end{itemize} 
In both cases, for $\lambda = 1+\beta-\theta_2 \in (\beta,1+\beta)$, the above analysis, \eqref{LEMEST1} and \eqref{PROPEQN5} give
\begin{equation}
\sum_{j=1}^{\infty} \sum_{k=1}^{\infty} \frac{jG(k) +kG(j)}{j+k} j^{\alpha}k^{\beta} w_j^l w_k^l \le 2 \sigma_p(G) M_1(w^l) M_{\lambda}(w^l) \le 2\rho_1\sigma_p(G) M_{\lambda}(w^l). \label{PROPEQN6} 
\end{equation}
Now gathering  \eqref{PROPEQN3}, \eqref{PROPEQN4} and \eqref{PROPEQN6}
\begin{equation*}
\frac{d}{dt} M_G^l + \sum_{j=1}^{\infty} \sum_{k=1}^{\infty} \sum_{i=1}^{j+k-1}\Big(\frac{G(j+k)}{j+k}- \frac{G(i)}{i}\Big) B_{j,k}^i i a_{j,k}^l w_j^l w_k^l \le 8A \sigma_p(G)\rho_1 M_{\lambda}(w^l).  
\end{equation*}

Now, let $T>0$ and $t\in[0,T]$. Integrating the above differential inequality over $[0,t]$, we find
\begin{align*}
M_G^l(t) +& \int_0^t\sum_{j=1}^{\infty} \sum_{k=1}^{\infty} \sum_{i=1}^{j+k-1}\Big(\frac{G(j+k)}{j+k}- \frac{G(i)}{i}\Big) B_{j,k}^i i a_{j,k}^l w_j^l(\tau) w_k^l(\tau) d\tau \nonumber\\
&\le M_G^l(0) + 8A \sigma_p(G)\rho_1 \int_0^T M_{\lambda}(w^l(\tau)) d\tau. 
\end{align*}
Since $\lambda<1+\beta$, the right hand side of the above inequality is finite and bounded uniformly with respect to $l\ge l_0$ by \eqref{FOM}, which completes the proof.
\end{proof}

\subsection{Proof of Theorem~\ref{TH1}}

The analysis performed in the previous section provides all information needed to control the behaviour of $w^l$ for large values of $i$. To guarantee	compactness, we need an estimate with respect to time which is already established in \cite[Lemma~3.4]{Laurencot 2001I} and we recall it now.

\begin{lemma}
Let $T\in(0,+\infty)$ and $i \geq 1$. There exists a constant $\Pi_i(T)$ depending only on $A$, $\rho_1$, $i$ and $T$ such that, for each $l \geq i$,
\begin{align}
\int_0^T \Big| \frac{dw_i^l}{dt}(t)\Big| dt \leq \Pi_i(T). \label{DERVBOUND}
\end{align}
\end{lemma}

We are now in a position to prove Theorem~\ref{TH1}. For that purpose we first recall a refined version of the de la Vall\'{e}e-Poussin theorem for integrable functions \cite[Theorem~7.1.6]{BLL 2019}.

\begin{theorem} \label{DlVPthm}
Let $(\Sigma,\mathcal{A}, \nu)$ be a measured space and consider a function $w \in L^1(\Sigma,\mathcal{A}, \nu)$. Then there exists a function $ G\in \mathcal{G}_{1,\infty}$  such that
\begin{align*}
G(|w|) \in L^1(\Sigma,\mathcal{A}, \nu).
\end{align*}
\end{theorem}

\begin{proof} [Proof of Theorem \ref{TH1}]
We apply Theorem~\ref{DlVPthm} with $\Sigma =\mathbb{N}$ and $\mathcal{A} = 2^{\mathbb{N}}$, the set of all subsets of $\mathbb{N}$. Defining the measure $\nu$ by
\begin{align*}
	\nu(J) =\sum_{i\in J} w_i^{\rm{in}},\qquad J \subset \mathbb{N},
\end{align*}
the condition $w^{in}\in Y_1^+$  ensures that $\zeta \mapsto \zeta$  belongs to $L^1(\Sigma,\mathcal{A}, \nu)$. By Theorem~\ref{DlVPthm} there is thus a function $G_0 \in \mathcal{G}_{1,\infty}$ such that $ G_0(w^{\rm{in}})$  belongs to $L^1(\Sigma,\mathcal{A}, \nu)$; that is,
\begin{align}
	\mathcal{G}_0 =\sum_{i=1}^{\infty} G_0(i) w_i^{\rm{in}} <\infty. \label{Gzero}
\end{align}
In the following, we denote by $C$ any positive constant depending only on $A$, $\alpha$, $\beta$, $\rho_0$, $\rho_1$, $G_0$ and $\mathcal{G}_0$. The dependence of $C$ upon additional parameters will be specified explicitly.

By~\eqref{LEMEST1} and~\eqref{DERVBOUND} the sequence $(w_i^l)_{l\ge i}$ is bounded in $W^{1,1}(0,T)$ for each $i\ge 1$ and $T\in(0,+\infty)$. We then infer from the Helly theorem \cite[pp.~372--374]{KF 1970} that there are a subsequence of $(w_i^l)_{l\ge l_0}$, still denoted by $(w_i^l)_{l\ge l_0}$, and a sequence $w=(w_i)_{i\ge 1}$ of functions of locally bounded variation such that
\begin{equation}
\lim_{l \to +\infty} w_i^l(t) = w_i(t) \label{wLIM}
\end{equation}
for each $i\ge 1$ and $t\ge 0$. Clearly $w_i(t) \ge 0$ for $i \ge 1$ and $t\ge 0$ and it follows from~\eqref{LEMEST1} and~\eqref{wLIM} that $w(t) \in Y_1^+$ with 
\begin{equation}
\|w(t)\|_{1} \le \|w^{\rm{in}}\|_{1}, \qquad  t\ge 0.\label{Y1NORMBOUND}
\end{equation}
  	Furthermore, as $G_0\in  \mathcal{G}_{1,\infty}$, we infer from \eqref{Gzero} and Proposition~\ref{prop1} that, for each $T \geq 0$ and $l\geq l_0$, there holds
 	\begin{equation}
 	\sum_{i=1}^{\infty} G_0(i) w_i^l(t) \leq C(T), \qquad t\in [0,T], \label{G_0bound1l}
 	\end{equation}
 	and
 	\begin{equation}
 	0 \leq \int_0^T\sum_{j=1}^{\infty} \sum_{k=1}^{\infty} \sum_{i=1}^{j+k-1}\Big(G_1(j+k)-G_1(i)\Big)i B_{j,k}^i  a_{j,k}^l w_j^l(s) w_k^l(s) ds \leq C(T), \label{G_0bound2l}
 	\end{equation}
 		where 
 	\begin{equation*}
 	G_1(\zeta) = \frac{G_0(\zeta)}{\zeta} ~~~ \text{for} ~~~\zeta \geq 0.
 	\end{equation*}
 	A consequence of \eqref{G_0bound2l} and the monotonicity properties of $G_1$ is that, for $i \geq 1$, $T\geq 0$ and $l \geq \max\{i+1,l_0\}$,
 	\begin{align*}
 	0 \leq \int_0^T\sum_{j=i+1}^{\infty} \sum_{k=1}^{j-1} \Big(G_1(j+k)-G_1(i)\Big)i B_{j,k}^i  a_{j,k}^l w_j^l(s) w_k^l(s) ds \leq C(T).
 	\end{align*}
 	Hence
 	\begin{align}
 	0 \leq \int_0^t\sum_{j=i+1}^{\infty} \sum_{k=1}^{j-1} \Big(G_1(j+k)-G_1(i)\Big) B_{j,k}^i  a_{j,k}^l w_j^l(s) w_k^l(s) ds \leq C(T). \label{G_0bound3l}
 	\end{align}

 	Due to~\eqref{wLIM}, we may let $l\to +\infty$ in~\eqref{G_0bound1l}, \eqref{G_0bound2l}, and~\eqref{G_0bound3l} and use Fatou's lemma to obtain 
\begin{equation}
 	\sum_{i=1}^{\infty} G_0(i) w_i(t) \leq C(T), \qquad t\in [0,T],\label{G_0bound1}
 	\end{equation}
\begin{equation}
 	0 \leq \int_0^T\sum_{j=1}^{\infty} \sum_{k=1}^{\infty} \sum_{i=1}^{j+k-1}\Big(G_1(j+k)-G_1(i)\Big)i B_{j,k}^i  a_{j,k} w_j(s) w_k(s) ds \leq  C(T),\label{G_0bound2}
 	\end{equation}
 	\begin{equation}
 0 \leq \int_0^T\sum_{j=i+1}^{\infty} \sum_{k=1}^{j-1} \Big(G_1(j+k)-G_1(i)\Big) B_{j,k}^i  a_{j,k} w_j(s) w_k(s) ds \leq C(T), \qquad i\ge 1.\label{G_0bound3}
 	\end{equation}
 	
Since $G_0\in \mathcal{G}_{1,\infty}$, it readily follows from \eqref{wLIM}, \eqref{G_0bound1l} and \eqref{G_0bound1} that 
\begin{align}
\lim_{l\to \infty} \|w^l(t) - w(t)\|_{1} =0, \qquad \text{for all}~~~ t\ge 0. \label{NORMCONV}
\end{align}
Hence, for $t\geq 0$
\begin{align}
\|w(t)\|_{1} =\lim_{l\to \infty} \|w^l(t)\|_{1} =\lim_{l\to \infty} \|w^{\rm{in},l}\|_{1} = \|w^{\rm{in}}\|_{1}, 
\end{align}
so that $w$ satisfies \eqref{MC}. Owing to \eqref{LEMEST1}, \eqref{wLIM}, \eqref{G_0bound1l}, \eqref{G_0bound3l}, \eqref{G_0bound1} and \eqref{G_0bound3}, we may then argue  exactly in the same way as in the proofs of \cite[Theorem~2.1]{MAP 23} and\cite[Theorem~3.1]{Laurencot 2001I} to show that $w$ is a solution to \eqref{NLDCBE}--\eqref{NLDCBEIC} on $[0,+\infty)$ in the sense of Definition~\ref{DEF1}.

We are left with proving the moment estimates~\eqref{ZEROMOM}, \eqref{MMOMTIM} and~\eqref{MMOMTIMI}. The last ones readily follow from~\eqref{FOM} and~\eqref{FOM1}, respectively, by~\eqref{wLIM} and a lower semicontinuity argument, whereas \eqref{ZEROMOM} is deduced from \eqref{LEMEST2} and \eqref{NORMCONV}.
\end{proof}

\section{Continuous dependence on Initial Data and Uniqueness}\label{SEC4}
This section is devoted to the proof of Theorem~\ref{TH2}. We begin with a preliminary result concerning continuous dependence for a suitable class of solutions.

\begin{prop} \label{CDP}
Assume that the kinetic coefficients satisfy~\eqref{LMC}, \eqref{AGROWTH} and \eqref{BCond}  and let $w$ and $\hat{w}$ be two solutions to \eqref{NLDCBE}--\eqref{NLDCBEIC} on $[0,+\infty)$ with initial conditions $w(0) = w^{\rm{in}}\in Y_{1+\beta}^+$  and $\hat{w}(0)= \hat{w}^{\rm{in}} \in Y_{1+\beta}^+$, respectively. Assume further that there are $T>0$ and $R>0$ such that
\begin{equation}
M_{1+\beta}(w(t))\le R,\qquad  M_{1+\beta}(\hat{w}(t))\le R , \qquad t\in[0,T]. \label{CDP1}
\end{equation}
Then there is $\Theta(T,R)>0$ depending only on $A$, $T$ and $R$ such that
\begin{equation}
\sup_{t\in[0,T]} \|w(t)-\hat{w}(t) \|_{1} \le \Theta(T,R) \|w(t)-\hat{w}(t) \|_{1}. \label{CDP2}
\end{equation}

\begin{proof}
For $i \ge 1$, we put $\xi_i = w_i - \hat{w}_i$ and $\varepsilon_i = \mathrm{sign}(\xi_i)$, where $\mathrm{sign}(h) =h/|h|$ if $h\in \mathbb{R}\setminus\{0\}$ and $\mathrm{sign}(0)=0$. Now, for $l\ge 2$ and  $t\in(0,T)$, we infer from~\eqref{IVOE} that
\begin{align}
\sum_{i=1}^l i |\xi_i(t)| = \sum_{i=1}^l i |\xi_i(0)| +\int_0^t \sum_{p=1}^3 \Upsilon_p^l(\tau) d\tau, \label{UPS}
\end{align}
where 
\begin{align*}
\Upsilon_1^l & =\frac{1}{2} \sum_{i=1}^l \sum_{j=i+1}^l\sum_{k=1}^{j-1}i \varepsilon_i B_{j-k,k}^i a_{j-k,k} (w_{j-k} w_k -\hat{w}_{j-k} \hat{w_k})\\
& \qquad\qquad - \sum_{i=1}^l \sum_{j=1}^{l-i}  i \varepsilon_i a_{i,j} (w_iw_j-\hat{w}_i \hat{w}_j)\\
&= \frac{1}{2} \sum_{i=1}^{l-1} \sum_{j=1}^{l-i} \Big(\sum_{s=1}^{i+j-1} s \varepsilon_s B_{i,j}^s- i\varepsilon_i -j\varepsilon_j\Big)a_{i,j} \big( w_i w_j - \hat{w}_i \hat{w}_j \big) \\
& = \frac{1}{2} \sum_{i=1}^{l-1} \sum_{j=1}^{l-i} \Big(\sum_{s=1}^{i+j-1} s \varepsilon_s B_{i,j}^s- i\varepsilon_i -j\varepsilon_j\Big)a_{i,j} \big( w_j + \hat{w}_j \big) \xi_i,
\end{align*}
\begin{align*}
\Upsilon_2^l = -\sum_{i=1}^l \sum_{j=l+1-i}^{\infty} i \varepsilon_ia_{i,j} (w_iw_j- \hat{w}_i \hat{w}_j),
\end{align*}
\begin{align*}
\Upsilon_3^l= \frac{1}{2}\sum_{i=1}^l \sum_{j=l+1}^{\infty} \sum_{k=1}^{j-1}i \varepsilon_iB_{j-k,k}^i a_{j-k,k} \big(w_{j-k,k} w_k- \hat{w}_{j-k} \hat{w}_k\Big).
\end{align*}
Noticing that, by \eqref{LMC},
\begin{align*}
\Big(\sum_{s=1}^{i+j-1} s\varepsilon_s B_{i,j}^s - i \varepsilon_i -&j \varepsilon_j\Big) \xi_i = \Big(\sum_{s=1}^{i+j-1} s\varepsilon_s \varepsilon_i B_{i,j}^s - i  -j \varepsilon_i \varepsilon_j\Big) |\xi_i|\\
&\le \Big(\sum_{s=1}^{i+j-1} s  B_{i,j}^s - i  +j  \Big) |\xi_i| \le 2j |\xi_i|.
\end{align*}
Consequently, using \eqref{AGROWTH} and \eqref{CDP1}, we obtain
\begin{align}
\big| \Upsilon_1^l \big|  \leq 2 A \Big(\sum_{j=1}^l j^{1+\beta} (w_j +\hat{w}_j) \Big) \sum_{i=1}^l i |\xi_i| \le 4AR\sum_{i=1}^{\infty} i |\xi_i|. \label{UPS1}
\end{align}

We next infer from~\eqref{AGROWTH}  that
\begin{align*}
\Big|\sum_{i=1}^l \sum_{j=l+1-i}^{\infty} i \varepsilon_i a_{i,j} w_i w_j \Big| & \le A \sum_{i=1}^l \sum_{j=l+1-i}^{\infty} \big(i^{1+\alpha}j^{\beta} + i^{1+\beta} j^{\alpha}\big) w_iw_j \\
& \le A \sum_{i=1}^l i^{1+\beta} w_i \sum_{j=l+1-i}^\infty j w_j,
\end{align*}
and it follows from \eqref{MC}, \eqref{CDP1} and Lebesgue's dominated convergence theorem that
\begin{align*}
\lim_{l\to + \infty} \int_0^t \Big| \sum_{i=1}^l \sum_{j=l+1-i}^{\infty} i \varepsilon_i a_{i,j} w_i(\tau) w_j(\tau) \Big| d\tau =0.
\end{align*}
We proceed similarly for $\hat{w}$ and conclude that 
\begin{align}
\lim_{l\to + \infty} \int_0^t \Upsilon_2^l(\tau) d\tau =0. \label{UPS2}
\end{align}

It finally follows from~\eqref{LMC} and~\eqref{AGROWTH} that
\begin{align*}
	& \left| \frac{1}{2}\sum_{i=1}^l \sum_{j=l+1}^{\infty} \sum_{k=1}^{j-1} i \varepsilon_iB_{j-k,k}^i a_{j-k,k} w_{j-k,k} w_k \right| \\
	& \qquad \le \frac{1}{2} \sum_{j=l+1}^{\infty} \sum_{k=1}^{j-1} \sum_{i=1}^l i B_{j-k,k}^i a_{j-k,k} w_{j-k,k} w_k \\
	& \qquad \le \frac{1}{2} \sum_{j=l+1}^{\infty} \sum_{k=1}^{j-1} j a_{j-k,k} w_{j-k,k} w_k \\
	& \qquad = \frac{1}{2} \sum_{k=1}^{\infty} \sum_{j=\max\{l+1,k+1\}}^{\infty}  j a_{j-k,k} w_{j-k,k} w_k \\
	& \qquad = \frac{1}{2} \sum_{k=1}^{\infty} \sum_{j=\max\{l-k+1,1\}}^{\infty}  (j+k) a_{j,k} w_j w_k \\
	& \qquad \le \frac{A}{2} \sum_{k=1}^{\infty} \sum_{j=\max\{l-k+1,1\}}^{\infty}  \left( j^{1+\alpha} k^\beta + j^{1+\beta} k^\alpha + j^\alpha k^{1+\beta} + j^\beta k^{1+\alpha} \right) w_j w_k \\
	& \qquad \le A \sum_{k=1}^{\infty} \sum_{j=\max\{l-k+1,1\}}^{\infty}  \left( j^{1+\beta} k + j k^{1+\beta} \right) w_j w_k \\
	& \qquad = A \sum_{k=1}^{l} \sum_{j=l-k+1}^{\infty}  \left( j^{1+\beta} k + j k^{1+\beta} \right) w_j w_k + A \sum_{k=l+1}^{\infty} \sum_{j=1}^{\infty}  \left( j^{1+\beta} k + j k^{1+\beta} \right) w_j w_k,
\end{align*}	
and we use once more~\eqref{MC}, \eqref{CDP1} and Lebesgue's dominated convergence theorem to conclude that	
\begin{equation*}
	\lim_{l\to\infty} \int_0^t \left| \frac{1}{2}\sum_{i=1}^l \sum_{j=l+1}^{\infty} \sum_{k=1}^{j-1} i \varepsilon_iB_{j-k,k}^i a_{j-k,k} w_{j-k,k}(\tau) w_k(\tau) \right| d\tau = 0.
\end{equation*}
As the same result is also true for $\hat{w}$, we end up with
\begin{align}
\lim_{l\to +\infty} \int_0^t \Upsilon_3^l(\tau) d\tau =0.\label{UPS3}
\end{align}
Owing to \eqref{UPS1}, \eqref{UPS2} and \eqref{UPS3} we may pass to the limit as $l\to +\infty$ in \eqref{UPS} and obtain 
\begin{align*}
\sum_{i=1}^{\infty} i |\xi_i(t)| \leq \sum_{i=1}^{\infty} i |\xi_i(0)| + 4AR \int_0^t \sum_{i=1}^{\infty} i |\xi_i(\tau)| d\tau.
\end{align*}
Applying the Gronwall lemma then completes the proof of Proposition~\ref{CDP}.
\end{proof}
\end{prop}

A first consequence of Proposition~\ref{CDP} is the well-posedness of \eqref{NLDCBE}--\eqref{NLDCBEIC} in $Y_{1+\beta}^+$.

\begin{corollary} \label{UCOR}
Assume that the kinetic coefficients satisfy \eqref{LMC}, \eqref{AGROWTH} and \eqref{BCond}. Given $w^{in} \in Y_{1+\beta}^+$ and $\rho_0 \le \rho_1$ such that
\begin{equation*}
\rho_0 \le M_0(w^{\rm{in}}) \le M_1(w^{\rm{in}})\le \rho_1,
\end{equation*}
 there is a unique mass-conserving solution $w$ to \eqref{NLDCBE}--\eqref{NLDCBEIC} on $[0,+\infty)$ satisfying
\begin{align*}
\sup_{t\in [0,T]}M_{1+\beta}(w(t)) <+\infty
\end{align*}
for each $T\in(0,+\infty)$. In addition,
\begin{align*}
\rho_0 \le M_0(w(t)) \le M_1((w(t))= M_1(w^{\rm{in}}) \le \rho_1.
\end{align*}
\end{corollary}
\begin{proof}
 The existence of at least one mass-conserving solutions to \eqref{NLDCBE}--\eqref{NLDCBEIC} on $[0,+\infty)$ with the properties stated in Corollary~\ref{UCOR} is given by Theorem~\ref{TH1} and its uniqueness is guaranteed by Proposition~\ref{CDP}. 
\end{proof}

\begin{proof}[Proof of Theorem~\ref{TH2}]
Given $w^{\rm{in}} \in \mathcal{S}(\rho_0, \rho_1, R)$, the unique mass-conserving solution $\Psi(\cdot, w^{in})$ to \eqref{NLDCBE}--\eqref{NLDCBEIC} on $[0,+\infty)$ given by Corollary~\ref{UCOR} remains in $\mathcal{S}(\rho_0, \rho_1, R)$ for all times according to Theorem~\ref{TH1} and Corollary~\ref{UCOR}, since $R\ge \mu_{1+\beta}$. As for the continuous dependence of $\Psi(\cdot, w^{\rm{in}})$ on the initial condition $w^{\rm{in}}$ in $\mathcal{S}(\rho_0, \rho_1, R)$, it is provided by Proposition~\ref{CDP}.
\end{proof}

\section{Stationary Solutions to \eqref{NLDCBE}--\eqref{NLDCBEIC}}\label{SEC5}

We now provide the proof of Theorem~\ref{TH3} which relies on a dynamical systems approach. First let us recall the following theorem from \cite[Proposition~22.13]{AMANN 90} or \cite[Proof of Theorem~5.2]{GPV 2004}.

\begin{theorem} \label{FPT}
Let $X$ be a Banach space, $Y$ be a subset of $X$, and $\Psi:[0,+\infty)\times Y\mapsto Y$ be a dynamical system with a positively invariant set $Z\subset Y$ which is a non-empty convex and compact subset of $X$. Then there is $x_0\in Z$ such that $\Psi(t,x_0)=x_0$ for all $t\ge 0$.
\end{theorem}

 \begin{proof}[Proof of Theorem~\ref{TH3}~(a)]
 We fix $0<\rho_0 < \rho_1$ and define 
 \begin{equation*}
 \mathcal{Z} := \Big\{ y \in Y_1^+: M_0(y) = \rho_0, \ M_1(y) =\rho_1, \ M_m(y) \le \mu_{m}+\rho_1,~~ m > 1 \Big\},
 \end{equation*}
 where $\mu_m$ is defined in Theorem~\ref{TH1}. Clearly $\mathcal{Z}$ is non-empty as it contains $(\rho_1 \delta_{i,1})_{i \ge 1}$ and $\mathcal{Z}$ is convex and compact in $Y_1$. In addition $\mathcal{Z} \subset  \mathcal{S}(\rho_0, \rho_1, \mu_{1+\beta}+\rho_1)$ and it follows from~\eqref{IVOE} and~\eqref{Z1} that, for any $w^{\rm{in}}\in \mathcal{Z}$,
 \begin{equation*}
 	M_0(\Psi(t,w^{\rm{in}})) = M_0(w^{\rm{in}}), \qquad t\ge 0.
 \end{equation*} 
 Consequently, combining Theorem~\ref{TH2} and the above identity implies that $\mathcal{Z}$ is an invariant set for  $\Psi$. Since  $\Psi$ is a dynamical system on $\mathcal{S}(\rho_0, \rho_1, \mu_{1+\beta}+\rho_1)$, it follows from Theorem~\ref{FPT} that there is $w^{\star}\in \mathcal{Z}$ such that $\Psi(t, w^{\star}) = w^{\star}$ for all $t\ge 0$, i.e, $w^{\star}$ is a stationary solution to~\eqref{NLDCBE}. Moreover $w^{\star}$ satisfies the properties listed in Theorem~\ref{TH3}~(a) since it belongs to $\mathcal{Z}$. 
 \end{proof}
 
We now turn to the second statement in Theorem~\ref{TH3} and begin with a preliminary result establishing a time monotonicity property of $M_0(w)-w_1$ for solutions to~\eqref{NLDCBE}--\eqref{NLDCBEIC} when the daughter distribution function $B$ satisfies~\eqref{Z3}.

\begin{lemma}\label{LemZ}
Assume that the kinetic coefficients $(a_{i,j})$ and $(B_{j,k}^i)$ satisfy~\eqref{ASYMM}, \eqref{LMC} and~\eqref{Z3}. Let $w^{\rm{in}}\in Y_1^+$ and consider a solution $w$ to \eqref{NLDCBE}--\eqref{NLDCBEIC} on $[0,+\infty)$ in the sense of Definition~\ref{DEF1}. Then
\begin{equation*}
	M_0(w(t)) - w_1(t) \ge M_0(w^{\rm{in}}) - w_1^{\rm{in}}, \qquad t\ge 0.
\end{equation*}
\end{lemma}

\begin{proof}
	On the one hand, it follows from~\eqref{IVOE} (for $i=1$) that, for $t> 0$, 
	\begin{align*}
		& w_1(t) - w^{\rm{in}} + \int_0^t \sum_{j=1}^\infty a_{1,j} w_1(\tau) w_j(\tau) d\tau \\
		& \qquad = \frac{1}{2} \int_0^t \sum_{j=2}^\infty \sum_{k=1}^{j-1} B_{j-k,k}^1 a_{j-k,k} w_{j-k}(\tau) w_k(\tau) d\tau \\
		& \qquad = \frac{1}{2} \int_0^t \sum_{k=1}^\infty \sum_{j=k+1}^{\infty} B_{j-k,k}^1 a_{j-k,k} w_{j-k}(\tau) w_k(\tau) d\tau \\
		& \qquad = \frac{1}{2} \int_0^t \sum_{j=1}^\infty \sum_{k=1}^{\infty} B_{j,k}^1 a_{j,k} w_j(\tau) w_k(\tau) d\tau \\
		& \qquad = \frac{1}{2} \int_0^t \left[ B_{1,1}^1 a_{1,1} w_1(\tau)^2 + \sum_{j=2}^\infty \sum_{k=2}^{\infty} B_{j,k}^1 a_{j,k} w_j(\tau) w_k(\tau) \right] d\tau \\
		& \qquad\qquad + \frac{1}{2} \int_0^t \left[ \sum_{j=2}^{\infty} B_{j,1}^1 a_{j,1} w_j(\tau) w_1(\tau) + \sum_{k=2}^{\infty} B_{1,k}^1 a_{1,k} w_1(\tau) w_k(\tau) \right] d\tau.
	\end{align*}
	Owing to~\eqref{LMC11} and the symmetry properties of $(B_{j,k}^i)$, we further obtain
	\begin{equation}
		\begin{split}
		w_1(t) - w^{\rm{in}} & = \int_0^t \sum_{j=2}^\infty \big( B_{j,1}^1 - 1 \big) a_{1,j} w_1(\tau) w_j(\tau) d\tau \\
		& \qquad + \frac{1}{2} \int_0^t \sum_{j=2}^\infty \sum_{k=2}^{\infty} B_{j,k}^1 a_{j,k} w_j(\tau) w_k(\tau) d\tau.
		\end{split} \label{Z8}
	\end{equation}
	On the other hand, we infer from~\eqref{LMC11}, \eqref{IVOE} and the symmetry properties of $(B_{j,k}^i)$ that, for $t>0$,
	\begin{align*}
		M_0(w(t)) - M_0(w^{\rm{in}}) & = \frac{1}{2} \int_0^t \sum_{j=1}^\infty \sum_{k=1}^\infty \left( \sum_{i=1}^{j+k-1} B_{j,k}^i - 2 \right) a_{j,k} w_j(\tau) w_k(\tau) d\tau \\
		& = \frac{1}{2}\int_0^t \big( B_{1,1}^1-2 \big) w_1(\tau)^2 d\tau \\
		& \qquad + \frac{1}{2} \int_0^t \sum_{k=2}^\infty \left( \sum_{i=1}^k B_{1,k}^i - 2 \right) a_{1,k} w_1(\tau) w_k(\tau) d\tau \\
		& \qquad + \frac{1}{2}\int_0^t \sum_{j=2}^\infty \left( \sum_{i=1}^j B_{j,1}^i - 2 \right) a_{j,1} w_j(\tau) w_1(\tau) d\tau \\
		& \qquad + \frac{1}{2} \int_0^t \sum_{j=2}^\infty \sum_{k=2}^\infty \left( \sum_{i=1}^{j+k-1} B_{j,k}^i - 2 \right) a_{j,k} w_j(\tau) w_k(\tau)  d\tau \\
		& = \int_0^t \sum_{j=2}^\infty \left( \sum_{i=1}^j B_{j,1}^i - 2 \right) a_{1,j} w_j(\tau) w_1(\tau) d\tau \\
		& \qquad + \frac{1}{2} \int_0^t \sum_{j=2}^\infty \sum_{k=2}^\infty \left( \sum_{i=1}^{j+k-1} B_{j,k}^i - 2 \right) a_{j,k} w_j(\tau) w_k(\tau)  d\tau.
	\end{align*}
	Substracting~\eqref{Z8} from the above identity leads us to
	\begin{align*}
		& M_0(w(t)) - M_0(w^{\rm{in}}) -w_1(t) + w^{\rm{in}} \\
		& \qquad = \int_0^t \sum_{j=2}^\infty \left( \sum_{i=2}^j B_{j,1}^i - 1 \right) a_{1,j} w_j(\tau) w_1(\tau) d\tau \\
		& \qquad + \frac{1}{2} \int_0^t \sum_{j=2}^\infty \sum_{k=2}^\infty \left( \sum_{i=2}^{j+k-1} B_{j,k}^i - 2 \right) a_{j,k} w_j(\tau) w_k(\tau)  d\tau,
	\end{align*}
	from which Lemma~\ref{LemZ} readily follows by~\eqref{Z3}.
\end{proof}

We are now in a position to complete the proof of Theorem~\ref{TH3}.

\begin{proof}[Proof of Theorem~\ref{TH3}~(b)]
	We fix $0<\rho_0<\rho_1$, $\eta\in (0,\min\{\frac{\rho_1}{2} , \rho_1-\rho_0\})$ and define
	 \begin{equation*}
		\mathcal{Z} := \left\{ y \in Y_1^+: 
		\begin{array}{l}
			M_0(y) \ge \max\{y_1 + \eta , \rho_0\}, \\ 
			M_1(y) =\rho_1, \\ 
			M_m(y) \le \mu_{m}+\rho_1 + 2^m,~~ m > 1 
		\end{array} \right\},
	\end{equation*}
	where $\mu_m$ is defined in Theorem~\ref{TH1}. First, the sequence $(\rho_1-2\eta,\eta,0,\cdots)$ belongs to $\mathcal{Z}$ so that it is non-empty, as well as convex and compact in $Y_1^+$. We next infer from Theorem~\ref{TH2} that $\mathcal{Z}$ is positively invariant for $\Psi$. We then argue as in the proof of Theorem~\ref{TH3}~(a) and conclude that there is a stationary solution $w^\star$ to~\eqref{NLDCBE} in $\mathcal{Z}$. In particular, $w^\star$ satisfies~\eqref{Z4} since $M_0(w^\star) \ge w_1^\star + \eta > w_1^\star$.
\end{proof}

\begin{remark}
	In principle, the stationary solution to~\eqref{NLDCBE} constructed in the previous proof depends on $\eta$ but it is unclear whether different choices of $\eta$ lead to different stationary solutions.
\end{remark}

 	\subsection*{Funding} The work of the authors is partially supported by the Indo-French Centre for Applied Mathematics (MA/IFCAM/19/58) within the project Collision-Induced Fragmentation and Coagulation: Dynamics and Numerics. AKG wishes to thank Science and Engineering Research Board (SERB), Department of Science and Technology (DST), India, for their funding support through the MATRICS  project MTR/2022/000530 for completing this work. MA would like to thank the University Grant Commission (UGC),
India for granting the Ph.D. fellowship through Grant No. 416611.
 	\subsection*{Acknowledgements}  Part of this work was done while PhL enjoyed the hospitality of Department of Mathematics, Indian Institute of Technology Roorkee, India.

\end{document}